\documentclass[a4paper,11pt]{article}
\usepackage{hyperref}
\usepackage{amsfonts,latexsym,rawfonts,amsmath,amssymb,amsthm,mathrsfs}
\usepackage{amsmath,amssymb,amsfonts,latexsym,lscape,rawfonts}

\usepackage[all]{xy}
\usepackage{eufrak}
\usepackage{graphicx,psfrag}

\usepackage{array,tabularx}

\usepackage{setspace}

\newtheorem{thm}{Theorem}[section]
\newtheorem{cor}[thm]{Corollary}

\newtheorem{prop}[thm]{Proposition}

\theoremstyle{remark}
\newtheorem{rmk}[thm]{Remark}

\theoremstyle{definition}

\title{The $L^{\frac{3}{2}}$-norm of the scalar curvature under the Ricci flow on a $3$-manifold}
\author{Hongnian Huang}
\date{}

\begin{document}
\maketitle
\begin{abstract}
Assume $M$ is a closed 3-manifold whose universal covering is not $S^3$. We show that the obstruction to extend the Ricci flow is the boundedness $L^{\frac{3}{2}}$-norm of the scalar curvature $R(t)$, i.e, the Ricci flow can be extended over finite time $T$ if and only if the $||R(t)||_{L^{\frac{3}{2}}}$ is uniformly bounded for $0 \leq t < T$ . On the other hand, if the fundamental group of $M$ is finite and the $||R(t)||_{L^{\frac{3}{2}}}$ is bounded for all time under the Ricci flow, then $M$ is diffeomorphic to a 3-dimensional spherical space-form. \\
\end{abstract}

\section{Introduction}
In recent years, there has been an increasing interest to understand the obstructions to extending the Ricci flow. For instance, N. Sesum\ \cite{Se} has shown that if the Riemannian curvature blows up at finite time $T$, then the Ricci curvature will also blow up at time $T$. The conjecture made by X. Chen is that the Ricci flow can be extended over time $T$ if and only if the scalar curvature $R(t)$ is uniformly bounded  at $[0, T)$. B. Wang\ \cite{Wa} proved that if the Ricci curvature has an uniformly lower bound in $[0,T)$ and $||R||_{\alpha, M \times [0,T)} < \infty, \alpha \geq \frac{n+2}{2}$, then the flow can be extended over time $T$. In this note, we improve the latter result when $M$ is a 3-dimensional manifold by proving the following:

\begin{thm}
\label{Th}
Assume that the unnormalized Ricci flow $g(t)$ on a closed 3-manifold $M$ blows up at time $T \leq +\infty$, and that the  
$$||R(t)||_{L^{\frac{3}{2}}} < C, \ 0 \leq t < T,$$ 
where $R(t)$ is the scalar curvature of $g(t)$. Then $M$ is diffeomorphic to a 3-dimensional spherical space-form.
\end{thm}

\begin{rmk}
We say that the Ricci flow $g(t)$ blows up at infinity if
$$
\lim_{t_0 \rightarrow \infty} \max_{x \in M, t \leq t_0} |Rm|(x,t) = \infty.
$$
\end{rmk}

As an immediate corollary, we have

\begin{cor}
Assume $M$ is a closed 3-manifold whose universal covering is not $S^3$, then the Ricci flow can be extended over finite time $T$ if and only if $||R(t)||_{L^{\frac{3}{2}}}$ of $M$ are uniformly bounded for $t < T$.
\end{cor}

Another corollary of our theorem is

\begin{cor}
\label{P}
Let $M$ be a closed 3-manifold with finite fundamental group and uniformly bounded $||R(t)||_{L^{\frac{3}{2}}}$ for $t < +\infty$, then $M$ is diffeomorphic to a 3-dimensional spherical space-form.
\end{cor}

Using the idea from our main theorem and Perelman's classifications of 3-dimensional non-compact $\kappa$-solutions, we give another proof of Hamilton's well-known result:

\begin{prop}
\label{Hamilton-Ricci}
A positive Ricci curvature compact 3-manifold must be diffeomorphic to $S^3$ or a quotient of it by a finite group.
\end{prop}

{\bf Acknowledgments:} The author would like to express his gratitude to X. Chen who brought this problem to his attention and provided many helpful and stimulating discussions. He is very grateful of V. Apostolov's detailed suggestions for this paper. He also would like to thank H. Li for discussing and reviewing the paper and R. Haslhofer for useful comments.

\section{Preliminaries}

Let $M$ be a closed 3-dimensional manifold, i.e, a compact manifold without boundary and $g_0$ be a smooth metric on $M$. We want to study the behavior of the metrics obtained by the so-called unnormalized Ricci flow equation: $$\frac{\partial g(t)}{\partial t} = -2 \mbox{Ric}(g(t)), \quad g(0)=g_0.$$ 

The short time existence of Ricci flow is established by R. Hamilton\ \cite{H3}, see also D. DeTurck\ \cite{De}. One important phenomenon of the flow is that it may form the singularities, i.e, there is a finite time $T$ such that the solution of Ricci flow equation cannot extend over time $T$. R. Hamilton\ \cite{HS} proved that the solution cannot extend over time $T$ precisely when $$\lim_{t \rightarrow T}\max_{x \in M}|Rm(x,T)| = \infty,$$ where $Rm(x, t)$ denotes the Riemannian curvature tensor of $g(t)$ at $x$. In this situation we say that the curvature blows up at $T$.

There are many examples to illustrate appearance of blow-ups. For instance, considering the evolution equation of the scalar curvature computed by R. Hamilton\ \cite{H3} $$\frac{\partial R}{\partial t} = \Delta R + 2 |Ric|^2,$$  we have the maximal principle and the Cauchy-Schwarz inequality, $$\frac{d}{d t}(R_{min}(t)) \geq \frac{2}{n}R_{min}^2(t) $$ in the sense of forward difference quotients. If at initial time $(M, g_0)$ is a manifold with positive scalar curvature, then the scalar curvature will blow up at finite time. Hence the curvature operator would also blow up. 

\section{Proof of Theorem (\ref{Th})}

\begin{proof}
Suppose the Ricci flow blows up at time $T \leq \infty$. Let us pick a sequence of points $(p_i, t_i)$ such that $t_i \rightarrow T$ as $i \rightarrow \infty$ and $|Rm(p_i,t_i)|=\max_{x \in M,t \leq t_i}{|Rm(x, t)|}$. Let $Q_i=|Rm(p_i,t_i)|, g_i(t)=Q_i g(Q_i^{-1}t+t_i)$. Using Perelman's celebrated $\kappa$-non-collapsing result\ \cite{P1}, the sequence $(M, g_i(0), p_i)$ converges to a $\kappa$-solution. Perelman\ \cite{P2} also classified all $\kappa$-solutions.

\begin{thm}{(Classification of $\kappa$-solution)}.

There is  $\bar{\epsilon}> 0$ such that the following is true for any $0 < \epsilon < \bar{\epsilon}$. There is $C = C(\epsilon)$ such that for any $\kappa > 0$ and any $\kappa$-solution $(M, g(t))$ one of the following holds.

\begin{enumerate}
\item $(M, g(0))$ is compact. In this case $M$ is diffeomorphic to a 3-dimensional spherical space-form.

\item $(M, g(0))$ is a noncompact manifold of positive curvature. All points of $(M, g(0))$ are either contained in the core of a $(C, \epsilon)$-cap or are the centers of a $\epsilon$-neck in $(M, g(0))$.

\item $(M, g(0))$ is isometric to the quotient of the product of a round $\mathbb{S}^2$ and $\mathbb{R}$ by a free, orientation-preserving involution.

\item $(M, g(0))$ is isometric to the product of a round $\mathbb{S}^2$ and $\mathbb{R}$.

\item $(M, g(0))$ is isometric to $\mathbb{RP}^2 \times \mathbb{R}$, where the metric on $\mathbb{RP}^2$ is of constant Gaussian curvature.

\end{enumerate}
\end{thm}

\begin{proof}
See e.g, Morgan-Tian\ \cite{MT}, Theorem (9.93). Also \ \cite{BBBBMP} \cite{CZ} \cite{KL}.
\end{proof}

Now let us complete the proof of our main theorem. Notice that $||R||_{L^{\frac{3}{2}}}$ is scaling invariant. If it is uniformly bounded for $t < T$, then the limited generalized manifold we have obtained cannot belong to the last 3 cases of the above theorem. That is because the $||R(0)||_{L^{\frac{3}{2}}}$ of the last 3 cases are all infinity. The limited generalized manifold cannot belong to the second case. The reason is that the second case contains infinite many disjoint $\epsilon$-necks whose $||R||_{L^{\frac{3}{2}}}$ is bounded below. So the limited manifold must be compact and diffeomorphic to our original manifold $M$.
\end{proof}

Next we give the proof of Corollary (\ref{P}).
\begin{proof}
There are two cases we need to consider:
\begin{itemize}
\item The Riemannian curvature blow up at $T \leq \infty$.

\item The Riemannian curvature are uniformly bounded for $t < \infty$.
\end{itemize}

For the first case, the conclusion of our corollary is obvious from the main theorem. For the second case, by shifting the initial time of the Ricci flow from $0$ to $T_i = i$, we obtain a sequence of Ricci flow $g_i(t)$. Since the Riemannian curvature are uniformly bounded, $g_i(t)$ converge to a compact $\kappa$-solution. Hence our manifold is diffeomorphic to a 3-dimensional spherical space-form.
\end{proof}

\section{Proof of Proposition (\ref{Hamilton-Ricci})}
\begin{proof}
For a compact 3-manifold $M$ with initial metric $g_0$ of positive Ricci curvature, we claim that the $||R||_{L^{\frac{3}{2}}}$ is bounded up to any singular time $T$. We will use the maximum principle for tensors established by R. Hamilton\ \cite{H4} to show this. Assume that $$m_1(x,t) \leq m_2(x,t) \leq m_3(x,t)$$ are three eigenvalues of the curvature operator $Rm$. Then the assumption of positive Ricci curvature corresponds to $$m_1(x,0) + m_2(x,0) \geq C_1.$$ Hence $m_3(x,0) > 0$ and there is a constant $C_2$ such that $$(m_1(x,0) + m_2(x,0)) / m_3(x,0) \geq C_2 $$ for all $x \in M$. Now let $Z$ be the subset of $Sym^2(\wedge^2 T_x^* M)$ such that for any element in $Z$, the sum of its two smallest eigenvalues satisfies $$m_1(x) + m_2(x) \geq C_1 $$ and $$ (m_1(x) + m_2(x)) / m_3(x) \geq C_2 $$ for all $x \in M$. 

It is easy to see that $Z$ is a closed subset and is invariant under parallel translation by the connection on $Sym^2(\wedge^2 T_x^* M)$ induced by the Levi-Civita connection on the tangent bundle of $TM$. Also $Z$ is a convex set. This is because if we assume $\mathcal{T}_0, \mathcal{T}_1 \in Z$, then for all $0 < t < 1$, 
\begin{eqnarray*}
& &m_{1, t \mathcal{T}_0 + (1-t) \mathcal{T}_1}(x) + m_{2, t \mathcal{T}_0 + (1-t) \mathcal{T}_1}(x) \\
&\geq& t (m_{1,\mathcal{T}_0}(x) + m_{2,\mathcal{T}_0}(x)) + (1 - t) (m_{1,\mathcal{T}_1}(x) + m_{2,\mathcal{T}_1}(x))
\end{eqnarray*}
and $$m_{3, t \mathcal{T}_0 + (1-t) \mathcal{T}_1} \leq t m_{3,\mathcal{T}_0}(x) + (1 - t) m_{3,\mathcal{T}_1}(x).$$ So $t \mathcal{T}_0 + (1-t) \mathcal{T}_1 \in Z$. The evolution formulas for the eigenvalues are
\begin{eqnarray*}
\frac{d}{dt}(m_1 + m_2) & = & m_1^2 + m_2^2 + m_3 (m_1 + m_2) \geq 0 \\
\frac{d}{dt} (m_3) & = & m_3^2 + m_1 m_2 > m_3^2 - m_2^2 \geq 0 \\
\frac{d}{dt} (\frac{m_1 + m_2}{m_3}) & = & \frac{m_1^2(m_3 - m_2) + m_2^2(m_3 - m_1)}{m_3^2} \geq 0
\end{eqnarray*}
Now we have $$m_1(x,t) + m_2(x,t) \geq C_1 $$ and $$(m_1(x,t) + m_2(x,t)) / m_3(x,t) \geq C_2 $$ up to the singular time $T$. If the energy is not bounded as $t \rightarrow T$, i.e, there is a sequence of time $T_i \rightarrow T$ such that $$\lim_{T_i \rightarrow T}||R||_{L^{\frac{3}{2}}}(T_i) = \infty.$$ Let us pick a sequence of point $(p_i, t_i)$ such that $t_i \leq T_i$ and $$|Rm(p_i,t_i)|=\max_{x \in M,t \leq t_i}{|Rm(x, t)|} \geq \max_{x \in M} {|Rm(x, T_i)|}.$$ Let $Q_i=|Rm(p_i,t_i)|, g_i(t)=Q_i g(Q_i^{-1}t+t_i)$. Notice that the volume of $(M, g_i(Q_i(T_i-t_i)))$ is greater than its $\left(||R||_{L^{\frac{3}{2}}}/6\right)^{\frac{3}{2}}$ which tends to infinity.  Since the scalar curvature is positive, we have 
$$
\mathrm{Vol} (M, g_i(0)) > \mathrm{Vol} (M, g_i(Q_i(T_i-t_i))).
$$
Hence the volume of $(M, g_i(0))$ approaches to infinity as $i \rightarrow \infty$. So this sequence of metrics must converge to a non-compact $\kappa$-solution. Either the universal covering of  this $\kappa$-solution is $S^2 \times \mathbb{R}$ or the $\kappa$-solution contains an $\epsilon$-neck. However, we can pick $\epsilon$ to be small enough to violate the property $$(m_1(x) + m_2(x)) / m_3(x) \geq C_2 .$$

So the energy must be bounded along the Ricci flow. We conclude that the limit $\kappa$-solution is a compact one. Using his maximum principle for tensors, R. Hamilton shows that the property $$m_3 - m_1 \leq K(m_1 + m_2 + m_3)^{1-\delta}$$ is preserved under the Ricci flow. Hence the compact $\kappa$-solution has positive constant sectional curvature at each point. By Schur's theorem, $M$ is diffeomorphic to $S^3$ or its quotient.
\end{proof}

\vskip3mm

Hongnian Huang

CIRGET (Inter-University Research Center for Geometry and Topology)

University of Quebec at Montreal, 

Montreal, Quebec, Canada.

hnhuang@gmail.com
\end{document}